\documentclass[11pt, notitlepage]{article}
\usepackage{amssymb,amsmath,comment}
\catcode`\@=11 \@addtoreset{equation}{section}
\def\thesection{\arabic{section}}

\def\theequation{\thesection.\arabic{equation}}
\catcode`\@=12
\usepackage{colortbl}%
\usepackage{a4wide}

\newcommand{\fa} {\forall}
\newcommand{\ds} {\displaystyle}

\newcommand{\al} {\alpha}

\newcommand{\de} {\delta}
\newcommand{\ga} {\gamma}

\newcommand{\Om} {\Omega}
\newcommand{\ra} {\rightarrow}

\newcommand{\De} {\Delta}
\newcommand{\la} {\lambda}

\newcommand{\noi} {\noindent}

\newcommand{\mb} {\mathbb}
\newcommand{\mc} {\mathcal}
\newcommand{\lra} {\longrightarrow}
\newcommand{\ld} {\langle}
\newcommand{\rd} {\rangle}

\usepackage[all]{xy}
\catcode`\@=11
\def\theequation{\@arabic{\c@section}.\@arabic{\c@equation}}
\catcode`\@=12

\def\proof{\noindent{\textbf{Proof. }}}
\def\QED{\hfill {$\square$}\goodbreak \medskip}

\newtheorem{Theorem}{Theorem}[section]
\newtheorem{Lemma}[Theorem]{Lemma}

\newtheorem{Corollary}[Theorem]{Corollary}

\begin{document}
{\vspace{0.01in}

\title
{A Nehari manifold for non-local elliptic operator with
concave-convex non-linearities and sign-changing weight function}

\author{
{\bf  Sarika Goyal\footnote{email: sarika1.iitd@gmail.com}} and {\bf  K. Sreenadh\footnote{e-mail: sreenadh@gmail.com}}\\
{\small Department of Mathematics}, \\{\small Indian Institute of Technology Delhi}\\
{\small Hauz Khaz}, {\small New Delhi-16, India}\\
 }

\date{}

\maketitle

\begin{abstract}

In this article, we study the existence and multiplicity of
non-negative solutions of following $p$-fractional equation:
$$ \quad \left\{
\begin{array}{lr}\ds
 \quad  - 2\int_{\mb R^n}\frac{|u(y)-u(x)|^{p-2}(u(y)-u(x))}{|x-y|^{n+p\al}}
dxdy = \la h(x)|u|^{q-1}u+ b(x)|u|^{r-1} u \; \text{in}\;
\Om \\
 \quad \quad \quad \quad u \geq 0 \; \mbox{in}\; \Om,\quad u\in W^{\al,p}(\mb R^n),\\
 \quad \quad\quad \quad\quad u =0\quad\quad \text{on} \quad \mb R^n\setminus \Om
\end{array}
\right.
$$
where $\Om$ is a bounded domain in $\mb R^n$, $p\geq 2$, $n> p\al$,
$\al\in(0,1)$, $0< q<p-1 <r < \frac{np}{n-ps}-1$, $\la>0$ and
$h$, $b$ are sign changing smooth functions. We show the existence
of solutions by minimization on the suitable subset of Nehari
manifold using the fibering maps. We find that there exists $\la_0$
such that for $\la\in (0,\la_0)$, it has at least two solutions.

\medskip

\noi \textbf{Key words:} Non-local operator, Fractional
$p-$Laplacian, sign-changing weight function, Nehari manifold,
fibering maps.

\medskip

\noi \textit{2010 Mathematics Subject Classification:} 35J35, 35J60,
35R11

\end{abstract}

\bigskip
\vfill\eject

\section{Introduction}
\setcounter{equation}{0}
 We consider the following $p$-fractional Laplace equation
$$ \quad \left\{
\begin{array}{lr}\ds
 \quad  - 2\int_{\mb R^n}\frac{|u(y)-u(x)|^{p-2}(u(y)-u(x))}{|x-y|^{n+p\al}}
dxdy = \la h(x)|u|^{q-1}u+ b(x)|u|^{r-1} u \; \text{in}\;
\Om \\
 \quad \quad \quad \quad u \geq 0 \; \mbox{in}\; \Om,\quad u\in W^{\al,p}(\mb R^n),\\
 \quad \quad\quad \quad\quad u =0\quad\quad \text{on} \quad \mb R^n\setminus \Om
\end{array}
\right.
$$
where $\Om$ is a bounded domain in $\mb R^n$ with Lipschitz
boundary, $p\geq 2$, $n> p\al$, $0< q< p-1 <r<
\frac{np}{n-ps}-1$, $\la>0$, $h$ and $b$ are sign changing
smooth functions.

Recently a lot of attention is given to the study of fractional and
non-local operators of elliptic type due to concrete real world
applications in finance, thin obstacle problem, optimization,
quasi-geostrophic flow etc. Dirichlet boundary value problem in case
of fractional Laplacian with polynomial type nonlinearity using
variational methods is recently studied in
\cite{tan,mp,var,weak,ls,yu}. Also existence and multiplicity
results for nonlocal operators with convex-concave type
nonlinearity is shown in \cite{mul}. In case of square root of
Laplacian, existence and multiplicity results with sign-changing
weight function with nonlinearity of the type $\la u+ b(x)
|u|^{\ga-1}$ is studied in \cite{yu}. In \cite{yu}, author also used
the idea of Caffarelli and Silvestre \cite{cs}, which gives a
formulation of the fractional Laplacian through Dirichlet-Neumann
maps. Eigenvalue problem related to $p-$Laplacian is also
studied in \cite{gf,pt}.

In particular, for $s=1$, a lot of work has been done for
multiplicity of positive solutions of semilinear elliptic problems
with positive nonlinearities \cite{ag, ABC, AAP, TA}. Moreover
multiplicity results with polynomial type nonlinearity with
sign-changing weight functions using Nehari manifold and fibering
map analysis is also studied in many papers(see refs.\cite{TA,
GA,DP,WU,WU6,WU9,WU10,WUFI, COA}. To the best of our knowledge there is
no work for non-local operator with convex-concave type nonlinearity
and sign changing weight functions. Here we use variational approach on Nehari manifold and fibering map analysis to solve the problem
\eqref{eq01}. In this paper, our work is motivated by the
work of Servadei \cite{mp}, Brown and Wu in \cite{WUFI}.

The aim of this paper is to study the existence and multiplicity of non-negative solutions for
the following equation driven by non-local operator $\mc L_{K}$ with
convex-concave type nonlinearities
\begin{equation}\label{eq01}
 \quad \left.
\begin{array}{lr}
 \quad -\mc L_{K}u = \la h(x) |u|^{q-1}u + b(x) |u|^{r-1}u \; \text{in}\;
\Om \\
 \quad \quad \quad \quad u = 0 \; \mbox{in}\; \mb R^n \setminus\Om.\\
\end{array}
\quad \right\}
\end{equation}
The non-local operator $\mc L_{K}$ is defined
as
\[\mc L_K u(x)= 2 \int_{\mb R^n}|u(y)- u(x)|^{p-2} (u(y)-u(x))K(x-y)
dy\;\;\text{for\; all}\;\; x\in \mb R^n,\]
where $K :\mb R^n\setminus\{0\}\ra(0,\infty)$ satisfying:\\
$(i)\; mK \in L^1(\mb R^n),\;\text{where}\; m(x) = \min\{1, |x|^p\},$\\
$(ii)$ there exist $\theta>0$ and $\al\in(0,1)$ such that $K(x)\geq
\theta |x|^{-(n+p\al)},$\\
$(iii)\; K(x) = K(-x)$ for any $ x\in \mb R^n\setminus\{0\}$.

\noi In case of $K(x) = |x|^{-(n+p\al)}$, $\mc L_K$ becomes
fractional $p$-Laplacian and denoted by $(-\De)^{s}_{p}$.

\noi More precisely, we study the problem to find $u\in
W^{\al,p}(\mb R^n)$ such that for every $v\in W^{\al,p}(\mb R^n)$,
\begin{align*}
\int_{Q}|u(x)-u(y)|^{p-2}(u(x)-u(y))(v(x)-v(y)) &K(x-y)dxdy= \la \int_{\Om} h(x)|u|^{q-1}u v dx \\
&\quad + \int_{\Om} b(x) |u|^{r-1} uv dx
\end{align*}
\noi holds. In our setting, it represents the weak formulation of
\eqref{eq01} (for this, we assume $(iii)$). We have the following
existence result:
\begin{Theorem}\label{t1}
There exists $\la_0>0$ such that for $\la\in(0,\la_0)$, \eqref{eq01}
admits at least two non-negative solutions.
\end{Theorem}
Here $\la_0$ is the maximum of $\la$ such that for $0<\la<\la_0$,
the fibering map $t\mapsto J_\la(tu)$  has exactly two critical
points
for each $u\in B^+\cap H^+$.\\

\noi Now we define the linear space $X$ as follows:
\[X= \left\{u|\;u:\mb R^n \ra\mb R \;\text{is measurable},
u|_{\Om} \in L^p(\Om)\;
 and\;  \left(u(x)- u(y)\right)\sqrt[p]{K(x-y)}\in
L^p(Q)\right\}\]

\noi where $Q=\mb R^{2n}\setminus(\mc C\Om\times \mc C\Om)$ and
 $\mc C\Om := \mb R^n\setminus\Om$. Moreover
 \[X_0=\{u\in X: u=0 \;\mbox{a.e. in}\;\mb R^n\setminus \Om\}.\]
\noi The paper is organized as follows: In section 1, we study the
properties of the spaces $X$ and $X_0$. In section 2, we introduce
Nehari manifold and study the behavior of Nehari manifold by
carefully analyzing the associated fibering maps. Section 3 contains
the existence of nontrivial solutions in $\mc N_{\la}^{+}$ and $\mc
N_{\la}^{-}$.
\section{Functional Analytic Settings}
In this section we prove some properties and results related to the
space $X$ and $X_0$.
\begin{Lemma}\label{l1}
 The spaces $X$ and $X_0$ are non-empty as $C_{c}^{2}(\Om)\subseteq
 X_{0}$. Such type of spaces were introduced for $p=2$ by Servadei \cite{mp}.
 \end{Lemma}

\begin{proof}Consider
\begin{align*}
\int_{\mb R^{2n}}|u(x)-u(y)|^{p}K(x-y)dx dy &= \left(\int_{\Om\times
\Om}+2 \int_{\Om\times \Om^{c}}\right)|u(x)-u(y)|^{p}K(x-y)dx dy\\
&\leq2 \int_{\Om\times\mb R^{n}}|u(x)-u(y)|^{p}K(x-y)dx dy
\end{align*}
Also we have,
\[|u(x)-u(y)|\leq 2\|u\|_{L^{\infty}(\mb R^{n})}, \quad |u(x)-u(y)|\leq \|\nabla u\|_{L^{\infty}(\mb R^{n})}|x-y|.\]
Thus
\[|u(x)-u(y)|\leq 2 \|u\|_{C^{1}(\mb R^{n})}\min \{1,|x-y|\}.\]
Hence
\begin{align*}
\int_{\mb R^{2n}}|u(x)-u(y)|^{p}K(x-y)dx dy &\leq 2^{p+1} \|u\|_{C^{1}(\mb R^{n})} \int_{\Om\times\mb R^{n}}m(x-y)K(x-y)dx dy\\
&\leq 2^{p+1} \|u\|_{C^{1}(\mb R^{n})} |\Om|\int_{\mb
R^{n}}m(z)K(z)dz<\infty,
\end{align*}
as required.\QED
\end{proof}
\noi Now, we define $W^{\al,p}(\Om)$, the usual fractional Sobolev
space endowed with the norm
\begin{align}\label{2}
\|u\|_{W^{\al,p}(\Om)}=\|u\|_{L^p}+ \left(\int_{\Om\times\Om}
\frac{|u(x)-u(y)|^{p}}{|x-y|^{n+p\al}}dxdy \right)^{\frac 1p}.
\end{align}
To study fractional Sobolev space in details we refer \cite{hic}.
\noi Now we consider the linear space
\[X= \left\{u|\;u:\mb R^n \ra\mb R \;\text{is measurable},
u|_{\Om} \in L^p(\Om)\;
 and\;  \left(u(x)- u(y)\right)\sqrt[p]{K(x-y)}\in
L^p(Q)\right\}\]
\noi where $Q=\mb R^{2n}\setminus(\mc C\Om\times
\mc C\Om)$ and
 $\mc C\Om := \mb R^n\setminus\Om$. The space X is normed linear space endowed with the norm
\begin{align}\label{1}
 \|u\|_X = \|u\|_{L^p(\Om)} +\left( \int_{Q}|u(x)- u(y)|^{p}K(x-y)dx
dy\right)^{\frac1p}.
\end{align}
It is easy to check that $\|.\|_X$ is a norm on $X$. For this we
first show that if $\|u\|_X=0$ then $u=0$ a.e. in $\mb R^n$. Indeed,
if $\|u\|_X=0$ then $\|u\|_{L^{p}(\Om)}=0$ which implies that
\begin{align}\label{e1}
u=0\;\mbox{ a.e in}\; \Om
\end{align}
and $ \int_{Q}|u(x)- u(y)|^{p}K(x-y)dx dy=0$. Thus $u(x)= u(y)$ a.e
in $Q$ means $u$ is constant in $Q$. Hence by \eqref{e1}, we have
$u=0$ a.e. in $\mb R^n$. Also triangle inequality follows from the
inequality $|a+b|^{p}\leq |a+b|^{p-1}|a|+|a+b|^{p-1}|b|$ $\fa$
$a,b\in \mb R$, $p\geq 1$ and H\"{o}lders inequality. moreover other
properties of norms are obvious. \QED
 We note that, if $K(x)=|x|^{n+p\al}$ then
norms in \eqref{2} and \eqref{1} are not same, because $\Om\times
\Om$ is strictly contained in $Q$.
\begin{Lemma}\label{l2}
Let $K :\mb R^n\setminus \{0\}\ra (0,\infty)$ be a function
satisfying $(ii)$. Then
\begin{enumerate}
\item[1.] If $u\in X$ then $u\in W^{\al, p}(\Om)$ and moreover
\[\|u\|_{W^{\al,p}(\Om)}\leq c(\theta) \|u\|_{X}.\]
\item[2.] If $u\in X_0$ then $u\in W^{\al, p}(\mb R^n)$ and moreover
\[\|u\|_{W^{\al,p}(\Om)}\leq \|u\|_{W^{\al,p}(\mb R^n)} \leq c(\theta) \|u\|_{X},\]
\end{enumerate}
 In both the cases $c(\theta)=\max\{1,
\theta^{-1/p}\}$, where $\theta$ is given in $(ii)$.
\end{Lemma}

\begin{proof}
\begin{enumerate}
\item[1.] Let $u\in X$, then by $(ii)$ we have
\begin{align*}
\int_{\Om\times\Om}\frac{|u(x)-u(y)|^{p}}{|x-y|^{n+p\al}} dx dy
&\leq \frac{1}{\theta}\int_{\Om\times\Om} |u(x)-u(y)|^{p} K(x-y) dx
dy\\
& \leq \frac{1}{\theta}\int_{Q} |u(x)-u(y)|^{p} K(x-y) dx dy<\infty.
\end{align*}
Thus
\begin{align*}
\|u\|_{W^{\al,p}}=\|u\|_{p}+
\left(\int_{\Om\times\Om}\frac{|u(x)-u(y)|^{p}}{|x-y|^{n+p\al}} dx
dy\right)^{\frac{1}{p}}\leq c(\theta) \|u\|_X.
\end{align*}
\item[2.] Let $u\in X_0$ then $u=0$ on $\mb R^n\setminus \Om$. So
$\|u\|_{L^{2}(\mb R^n)}= \|u\|_{L^{2}(\Om)}$. Hence
\begin{align*}
\int_{\mb R^{2n}}\frac{|u(x)-u(y)|^{p}}{|x-y|^{n+p\al}} dx dy &=
\int_{Q} \frac{|u(x)-u(y)|^{p}}{|x-y|^{n+p\al}} dx dy\\
& \leq \frac{1}{\theta}\int_{Q} |u(x)-u(y)|^{p} K(x-y) dx dy<
+\infty,
\end{align*}
\end{enumerate}
as required.\QED
\end{proof}
\begin{Lemma}\label{l3}
Let $K :\mb R^n\setminus \{0\}\ra (0,\infty)$ be a function
satisfying $(ii)$. Then there exists a positive constant $c$
depending on $n$ and $\al$ such that for every $u\in X_0$, we have
\[\|u\|_{L^{p^*}(\Om)}^{p}= \|u\|_{L^{p^*}(\mb R^n)}^{p}\leq c\int_{\mb R^{2n}}\frac{|u(x)-u(y)|^{p}}{|x-y|^{n+p\al}} dx dy, \]
where $p^*=\frac{np}{n-ps}$ is fractional critical Sobolev exponent.
\end{Lemma}
\begin{proof}
Let $u\in X_0$ then by Lemma \ref{l2}, $u\in W^{\al,p}(\mb R^n)$ and
also we know that $W^{\al,p}(\mb R^n)\hookrightarrow L^{p^*}(\mb
R^n)$ (see \cite{hic}). Then we have,
\[\|u\|_{L^{p^*}(\Om)}^{p}= \|u\|_{L^{p^*}(\mb R^n)}^{p}\leq c\int_{\mb R^{2n}}\frac{|u(x)-u(y)|^{p}}{|x-y|^{n+p\al}} dx dy,  \]
and hence the result. \QED
\end{proof}
 \begin{Lemma}\label{l4}
Let $K :\mb R^n\setminus \{0\}\ra (0,\infty)$ be a function
satisfying $(ii)$. Then there exists $C>1$, depending only on $n$,
$\al$, $p$, $\theta$ and $\Om$ such that for any $u\in X_0$,
\[\int_{Q} |u(x)-u(y)|^{p} K(x-y) dx dy\leq \|u\|_{X}^{p} \leq C\int_{Q} |u(x)-u(y)|^{p} K(x-y) dx dy.\]
i.e. \begin{align}\label{ee1} \|u\|_{X_0}^{p} = \int_{Q}
|u(x)-u(y)|^{p} K(x-y) dx dy
\end{align}
 is a norm on $X_0$ and equivalent to the norm on $X$.
\end{Lemma}
\begin{proof}
Clearly $\|u\|_{X}^{p}\geq \int_{Q} |u(x)-u(y)|^{p} K(x-y) dx dy$
and moreover, By Lemma \ref{l3}, $(ii)$ and using the embedding $
L^{p^*}(\Om)\hookrightarrow L^{p}(\Om)$, where
$p^*=\frac{np}{n-ps}$, we get
\begin{align*}
\|u\|_{X}^{p}&= \left(\|u\|_{p} + \left(\int_{Q}
|u(x)-u(y)|^{p} K(x-y) dx dy \right)^{1/p}\right)^p \\
&\leq 2^{p-1} \|u\|_{p}^{p} + 2^{p-1} \int_{Q} |u(x)-u(y)|^{p} K(x-y) dx dy\\
&\leq  2^{p-1} |\Om|^{1-\frac{p}{p^*}} \|u\|_{p^*}^{p} + 2^{p-1} \int_{Q} |u(x)-u(y)|^{p} K(x-y) dx dy\\
&\leq  2^{p-1} \;c|\Om|^{1-\frac{p}{p^*}} \int_{\mb R^{2n}}\frac{|u(x)-u(y)|^{p}}{|x-y|^{n+p\al}}dx dy + 2^{p-1} \int_{Q} |u(x)-u(y)|^{p} K(x-y) dx dy\\
&\leq 2^{p-1}
\left(\frac{c|\Om|^{1-\frac{p}{p^*}}}{\theta}+1\right)\int_{Q}
|u(x)-u(y)|^{p} K(x-y) dx dy\\
&= C \int_{Q} |u(x)-u(y)|^{p} K(x-y) dx dy,
\end{align*}
where $C>1$ as required. Now  we show that \eqref{ee1} is norm on
$X_0$. For this we need only to show that if $\|u\|_{X_0}=0$ then
$u=0$ a.e. in $\mb R^n$. Indeed, if $\|u\|_{X_0}=0$ then
 $ \int_{Q}|u(x)- u(y)|^{p}K(x-y)dx dy=0$ implies that $u(x)=
u(y)$ a.e in $Q$. Therefore, $u$ is constant in $Q$ and hence $u=c\in \mb R$
a.e in $\mb R^n$ and by definition of $X_0$, we have $u=0$ on $\mb
R^n\setminus \Om$. Thus $u=0$ a.e. in $\mb R^n$.\QED

\end{proof}
\begin{Lemma}
The space $(X_0, \|.\|_{X_0})$ is a reflexive Banach space.
\end{Lemma}
\begin{proof}
Let $\{u_k\}$ be a Cauchy sequence in $X_0$. Then by Lemma \ref{l3},
$\{u_k\}$ is Cauchy sequence in $L^{p}(\Om)$ and so $\{u_k\}$ has a
convergent subsequence. Thus we assume $u_k\ra u$ strongly in
$L^{p}(\Om)$. Since $u_k=0$ in $\mb R^n\setminus \Om$, we define
$u=0$ in $\mb R^n\setminus \Om$ and then $u_k\ra u$ strongly in
$L^{p}(\mb R^n)$ as $k\ra \infty$. Thus there exists a subsequence
still denoted by $u_{k}$ such that $u_{k}\ra u$ a.e. in $\mb R^n$.
Therefore one can easily show by Fatou's Lemma and using the fact that
$u_k$ is a Cauchy sequence that $u\in X_0$. Moreover, using the
same fact one can verify that $\|u_k -u\|_{X_0}\ra 0$ as $k\ra
\infty$. Hence $X_0$ is a Banach
space. Reflexivity of $X_0$ follows from the fact that $X_0$ is a
closed subspace of reflexive Banach space $W^{\al,p}(\mb R^n)$.\QED
\end{proof}
\noi Thus we have
 \[ X_0 = \{u\in X : u = 0 \;\text{a.e. in}\; \mb R^n\setminus \Om\}\]
with the norm
\begin{align}\label{01}
 \|u\|_{X _0}=\left(\int_{Q}|u(x)-u(y)|^{p}K(x-y)dx dy\right)^{\frac1p}
\end{align}
is a reflexive Banach space. Note that the norm $\|.\|_{X_0}$
involves the interaction between $\Om$ and $\mb R^n\setminus\Om$.
\begin{Lemma}
Let $K :\mb R^n\setminus \{0\}\ra (0,\infty)$ be a function
satisfying $(ii)$ and let $\{u_k\}$ be a bounded sequence in $X_0$.
Then, there exists $u\in L^{m}(\mb R^n)$ such that up to a
subsequence, $u_k\ra u$ in $L^{m}(\mb R^n)$ as $k\ra \infty$ for any
$m\in [1, p^*)$.
\end{Lemma}

\begin{proof}
As $\{u_k\}$ is bounded in $X_0$,  by Lemmas \ref{l2} and \ref{l4},
$\{u_k\}$ is bounded in $W^{\al,p}(\Om)$ and also in $L^{p}(\Om)$. Then
by assumption on $\Om$ and [4, Corollary 7.2], there exists $u\in
L^{m}(\Om)$ such that up to a subsequence $u_k\ra u$ in $
L^{m}(\Om)$ as $k\ra\infty$ for any $m\in[1,p^*)$. Since $u_k=0$ on
$\mb R^n\setminus \Om$, we can define $u:=0$ in $\mb R^n\setminus
\Om$ and we get $u_k\ra u$ in $ L^{m}(\mb R^n)$.\QED
\end{proof}
\section{Nehari manifold and Fibering map analysis for \eqref{eq01}}
\setcounter{equation}{0} \noi The Euler functional $J_{\la}: X_0 \ra
\mb R$ associated to the problem \eqref{eq01} is defined as
\[J_{\la}(u)= \frac{1}{p}\int_{Q}|u(x)-u(y)|^p K(x-y) dxdy -\frac{\la}{q+1}\int_{\Om}h(x)|u|^{q+1} dx -\frac{1}{r+1}\int_{\Om} b(x) |u|^{r+1}.\]
 Then $J_{\la}$ is Fr$\acute{e}$chet differentiable  and
\begin{align*}
\ld J_{\la}^{\prime}(u),v
\rd=\int_{Q}&|u(x)-u(y)|^{p-2}(u(x)-u(y))(v(x)-v(y)) K(x-y)dxdy\\
&\quad- \la \int_{\Om} h(x)|u|^{q-1} u v dx -  \int_{\Om}b(x)
|u|^{r-1}u v dx,
\end{align*}
which shows that the weak solutions of \eqref{eq01} are
critical points of the functional $J_{\la}$.

\noi It is easy to see that the energy functional $J_{\la}$ is not
bounded below on the space $X_0$, but is bounded below on an
appropriate subset of $X_0$ and a minimizer on subsets of this set
gives raise to solutions of \eqref{eq01}. In order to obtain the
existence results, we introduce the Nehari manifold
\begin{equation*}
\mc N_{\la}:= \left\{u\in X_0: \ld J_{\la}^{\prime}(u),u\rd=0
\right\},
\end{equation*}
where $\ld\;,\; \rd$ denotes the duality between $X_0$ and its dual
space. Therefore $u\in \mathcal N_{\la}$ if and only if
\begin{equation}\label{eq2}
\int_{Q} |u(x)-u(y)|^p K(x-y) dxdy - \la \int_{\Om}h(x) |u|^{q+1}
dx- \int_{\Om} b(x)|u|^{r+1}dx =0 .
\end{equation}
We note that $\mathcal N_{\la}$ contains every non zero solution of
\eqref{eq01}. Now as we know that the Nehari manifold is closely
related to the behavior of the functions $\phi_u: \mb R^+\ra \mb R$
defined as $\phi_{u}(t)=J_{\la}(tu)$. Such maps are called fiber
maps and were introduced by Drabek and Pohozaev in \cite{DP}. For
$u\in X_0$, we have
\begin{align*}
\phi_{u}(t) &= \frac{t^p}{p} \|u\|_{X_0}^p- \frac{\la
 t^{q+1}}{q+1}\int_{\Om} h(x)|u|^{q+1} dx - \frac{t^{r+1}}{r+1}\int_{\Om} b(x) |u|^{r+1} dx ,\\
\phi_{u}^{\prime}(t) &= t^{p-1}\|u\|_{X_0}^{p}- {\la
 t^{q}}\int_{\Om} h(x) |u|^{q+1} dx  - t^r\int_{\Om} b(x) |u|^{r+1} dx,\\
\phi_{u}^{\prime\prime}(t) &= (p-1)t^{p-2}\|u\|_{X_0}^p- q \la
 t^{q-1} \int_{\Om} h(x) |u|^{q+1} dx - r t^{r-1}\int_{\Om} b(x) |u|^{r+1} dx.
\end{align*}
Then it is easy to see that $tu\in \mathcal N_{\la}$ if and only if
$\phi_{u}^{\prime}(t)=0$ and in particular, $u\in \mc N_{\la}$ if
and only if $\phi_{u}^{\prime}(1)=0$. Thus it is natural to split
$\mathcal N_{\la}$ into three parts corresponding to local minima,
local maxima and points of inflection. For this we set
\begin{align*}
\mathcal N_{\la}^{\pm}&:= \left\{u\in \mc N_{\la}:
\phi_{u}^{\prime\prime}(1)
\gtrless0\right\} =\left\{tu\in X_0 : \phi_{u}^{\prime}(t)=0,\; \phi_{u}^{''}(t)\gtrless  0\right\},\\
\mathcal N_{\la}^{0}&:= \left\{u\in \mc N_{\la}:
\phi_{u}^{\prime\prime}(1) = 0\right\}=\left\{tu\in X_{0} :
\phi_{u}^{\prime}(t)=0,\; \phi_{u}^{''}(t)= 0\right\}.
\end{align*}

\noi Before studying the behavior of Nehari manifold using fibering
maps, we introduce some notations
\begin{align*}
B^{\pm}:=\{u\in X_0: \int_{\Om} b(x)|u|^{r+1} dx\gtrless 0\}, &\;\; B_{0}:=\{u\in X_0: \int_{\Om} b(x)|u|^{r+1} dx= 0\},\\
H^{\pm}:=\{u \in X_0: \int_{\Om} h(x)|u|^{q+1} dx \gtrless 0\},
&\;\; H_{0}:=\{u\in X_0: \int_{\Om} h(x)|u|^{q+1} dx= 0\},
\end{align*}
 and $H^{\pm}_{0}: =H^{\pm}\cup H_{0}$, $B^{\pm}_{0}: =
B^{\pm}\cup B_{0}$.

\noi Now we study the fiber map $\phi_{u}$ according to the sign of
$\int_{\Om} h(x)|u|^{q+1} dx$ and $\int_{\Om} b(x)|u|^{r+1} dx$.

\noi Case 1: $u\in H^{-}\cap B^{-}$.\\
In this case $\phi_{u}(0)=0$, $\phi_{u}^{\prime}(t)>0$ $\fa$ $t>0$
which implies that $\phi_{u}$ is strictly increasing and hence no
critical point.

\noi Case 2: $u\in H^{-}\cap B^{+}$.\\
In this case, firstly we define $m_{u}: \mb R^{+} \lra \mb R$ by
\begin{equation*}
m_{u}(t)= t^{p-1-q}\|u\|_{X_0}^p-  t^{r-q}\int_{\Om} b(x) |u|^{r+1}
dx.
\end{equation*}
Clearly, for $t>0$, $tu\in \mc N_{\la}$ if and only if $t$ is a
solution of
 \[m_{u}(t)={\la} \int_{\Om} h(x) |u|^{q+1} dx.\]
 As we have $m_{u}(t)\ra -\infty$ as $t\ra \infty$ and
\begin{align*}\label{eq01}
m_{u}^{\prime}(t)&= (p-1-q) t^{p-2-q} \|u\|_{X_0}^p -
(r-q)t^{r-1-q}\int_{\Om} b(x) |u|^{r+1} dx.
\end{align*}
Therefore $m_{u}^{\prime}(t)>0$ as $t\ra 0$. Since $u\in H^{-}$,
there exists $t_*(u)$ such that $m_{u}(t_*)={\la}\int_{\Om}
h(x)|u|^{q+1} dx .$ Thus for $0<t<t_*$, $\phi_{u}^{'}(t)=
t^{q+2}(m_{u}(t)-{\la} \int_{\Om} h(x) |u|^{q+1} dx )>0$ and for
$t>t_*$, $\phi_{u}^{'}(t)<0$. Hence $\phi_u$ is increasing on
$(0,t_*)$, decreasing on $(t_*, \infty)$. Since $\phi_{u}(t)>0$ for
$t$ close to $0$ and $\phi_{u}(t)\ra -\infty$ as $t\ra \infty$, we
get $\phi_{u}$ has exactly one critical point $t_{1}(u)$, which is a
global maximum point. Hence $t_{1}(u)u \in \mc N^{-}_{\la}$.

\noi Case 3: $u\in H^{+}\cap B^{-}$.\\
In this case $m_{u}(0)=0$, $m_{u}^{\prime}(t)>0$ $\fa$ $t>0$ which
implies that $m_{u}$ is strictly increasing and since  $u\in
H^+$, there exists a unique $t_1=t_{1}(u)>0$ such that
$m_u(t_1)=\la\int_{\Om}h(x) |u|^{q+1} dx$. This implies that
$\phi_u(t)$ is decreasing on $(0,t_1)$, increasing on $(t_1,
\infty)$ and $\phi_{u}^{\prime}(t_1)=0$. Thus $\phi_{u}$ has exactly
one critical point $t_{1}(u)$, corresponding to global minimum
point. Hence $t_{1}(u)u \in \mc N^{-}_{\la}$.

\noi Case 4: $u\in H^{+}\cap B^+$.

\noi In this case, we claim that there exists $\la_0>0$
such that for $\la\in(0,\la_0)$, $\phi_u$ has exactly two
critical points $t_1(u)$ and $t_2(u)$.
 Moreover, $t_1(u)$ is a local minimum
point and $t_2(u)$ is a local maximum point. Thus $t_1(u)u\in\mc
N_{\la}^{+}$ and $t_2(u)u\in\mc N_{\la}^{-}$.

We prove this claim in the following Lemma:
\begin{Lemma}\label{le02}
There exists $\la_0>0$ such that $\la<\la_0$, $\phi_u$ takes
positive value for all non-zero $u\in X_0$.  Moreover, if $\la<\la_0$ and $u\in H^{+}\cap B^+$ then $\phi_u$ has
exactly two critical points.
\end{Lemma}

\begin{proof}
Let $u\in X_0$ and $\int_{\Om} b(x)|u|^{r+1}dx >0$, define
\[F_{u}(t):= \frac{t^p}{p}\int_{Q}|u(x)-u(y)|^p K(x-y)dxdy -\frac{t^{r+1}}{r+1}\int_{\Om}b(x)|u|^{r+1}dx. \]
Then
\[F_{u}^{\prime}(t)= t^{p-1} \int_{Q}|u(x)-u(y)|^p K(x-y)dxdy - t^r \int_{\Om}b(x)|u|^{r+1}dx \]
and $F_{u}$ attains its maximum value at $t_*=
\left(\frac{\int_{Q}|u(x)-u(y)|^p
K(x-y)dxdy}{\int_{\Om}b(x)|u|^{r+1}dx}\right)^{\frac{1}{r-p+1}}.$ Moreover, \[F_u(t^*)=
\left(\frac{1}{p}-\frac{1}{r+1}\right)\left[\frac{(\int_{Q}|u(x)-u(y)|^p
K(x-y)dxdy)^{r+1}}{(\int_{\Om}b(x)|u|^{r+1}dx)^p}\right)^{\frac{1}{r-p+1}}\]
and\[F_{u}^{\prime\prime}(t^*)= (p-r-1)\frac{(\int_{Q}|u(x)-u(y)|^p
K(x-y)dxdy)^{\frac{r-1}{r-p+1}}}{\left(\int_{\Om}b(x)|u|^{r+1}dx\right)^{\frac{p-2}{r-p+1}}}<0.\]

\noi Let $S_{r+1}$ be the Sobolev constant of embedding
$W^{\al,p}(\mb R^n)\hookrightarrow L^{r+1}(\mb R^n)$, then
\[\|u\|_{L^{r+1}(\Om)}=\|u\|_{L^{r+1}(\mb R^n)}\leq S_{r+1}\|u\|_{W^{\al,p}(\mb R^n)}.\]
By Lemmas \ref{l2} and \ref{l4},
\[\|u\|_{W^{\al,p}(\mb R^n)}\leq C(\theta)\|u\|_{X_0}= M \|u\|_{X_0}.\]
Combining above two inequalities we get,
\[\frac{1}{ (M S_{r+1})^{p(r+1)}}\leq \frac{(\int_{Q}|u(x)-u(y)|^p
K(x-y)dxdy)^{r+1}}{(\int_{\Om}|u|^{r+1}dx)^p}.\]

\noi Hence \[F_{u}(t^*)\geq
\frac{r-p+1}{p(r+1)}\left(\frac{1}{\|b^{+}\|^{p}_{\infty} (M
S_{r+1})^{p(r+1)}}\right)^{\frac{1}{r-p+1}}= \de,\] which is
independent of $u$. We now show that there exists $\la_0>0$ such that
$\phi_u(t^*)>0$. Using  Sobolev embedding of fractional spaces we get,
\begin{align*}
\frac{t_{*}^{q+1}}{q+1}  &\int_{\Om} h(x)|u|^{q+1}dx \\
&\leq \frac{1}{q+1} \|h\|_{\infty}
(MS_{q+1})^{q+1}\left(\frac{\int_{Q}|u(x)-u(y)|^p
K(x-y)dxdy}{\int_{\Om}b(x)|u|^{r+1}dx}\right)^{\frac{q+1}{r-p+1}}\|u\|^{q+1}\\
&= \frac{1}{q+1} \|h\|_{\infty}
(MS_{q+1})^{q+1}\left[\frac{(\int_{Q}|u(x)-u(y)|^p
K(x-y)dxdy)^{r+1}}{(\int_{\Om}b(x)|u|^{r+1}dx)^p}\right]^{\frac{q+1}{p(r-p+1)}}\\
&=\frac{1}{q+1} \|h\|_{\infty}
(MS_{q+1})^{q+1}\left(\frac{p(r+1)}{r-p+1}\right)^{\frac{q+1}{p}}
\left(F_{u}(t_*)\right)^{\frac{q+1}{p}}= c
F_{u}(t_*)^{\frac{q+1}{p}},
\end{align*}
where $c$ is a constant independent of $u$. Thus
\begin{align*}
\phi_{u}(t_*)&\geq F_u(t_*)-\la c F_{u}(t_*)^{\frac{q+1}{p}}
 = F_{u}(t_*)^{\frac{q+1}{p}}(F_u(t_*)^{\frac{p-1-q}{p}}-\la c)
 \geq \de^{\frac{q+1}{p}} (\de^{\frac{p-1-q}{p}} -\la c).
\end{align*}
Let $\la<\frac{\de^{\frac{p-1-q}{p}}}{c}=\la_0$. Then choice of such
$\la$ completes the proof.
 \QED
\end{proof}
\begin{Corollary}\label{co01}
If $\la<\la_0$, then there exists $\de_1>0$ such that
$J_{\la}(u)\geq \de_1$ for all $u\in \mc N_{\la}^{-}$.
\end{Corollary}

\begin{proof}
Let $u\in \mc N_{\la}^{-}$, then $\phi_{u}$ has a positive global
maximum at $t=1$ and $\int_{\Om} h(x)|u|^{q+1}dx >0$. Thus
\begin{align*}
J_{\la}(u)=\phi_u(1)&= \phi_{u}(t_*)\\
&\geq F_{u}(t_*)^{\frac{q+1}{p}}(F_u(t_*)^{\frac{p-1-q}{p}}-\la c)\\
 &\geq \de^{\frac{q+1}{p}} (\de^{\frac{p-1-q}{p}} -\la c)>0 \;\mbox{if}\;
 \la<\la_0,
 \end{align*}
 where $\de$ is same as in Lemma \ref{le02}, and hence the
 result.\QED
\end{proof}
\begin{Corollary}
If $0<\la<\la_0$, then $\mc N_{\la}^{0}=\emptyset$.
\end{Corollary}

\noi In the following lemma we show that the minimizers on subsets of
$\mc N_{\la}$ are  solutions of $\eqref{eq01}$.
 \begin{Lemma}\label{le30}
Let $u$ be a local minimizer for $J_{\la}$ on subsets $\mc{N_\la}^{+}$ or $\mc{N_{\la}}^{-}$ of $\mc
N_{\la}$ such that $u\notin \mc N_{\la}^{0}$, then $u$ is a critical
point for $J_{\la}$.
\end{Lemma}
\proof Since  $u$ is a minimizer for $J_{\la}$
under the constraint $I_{\la}(u):=\ld J_{\la}^{\prime}(u),u\rd =0$, by the theory of Lagrange multipliers, there exists $\mu \in
\mb R$ such that $ J _{\la}^{\prime}(u)= \mu I_{\la}^{\prime}(u)$.
Thus $\ld J_{\la}^{\prime}(u),u\rd= \mu\;\ld
I_{\la}^{\prime}(u),u\rd = \mu \phi_{u}^{\prime\prime}(1)$=0, but
$u\notin \mc N_{\la}^{0}$ and so $\phi_{u}^{\prime\prime}(1) \ne 0$.
Hence $\mu=0$ completes the proof.\QED
\begin{Lemma}
$J_{\la}$ is coercive and bounded below on $\mc N_{\la}$.
\end{Lemma}

\begin{proof}On $\mc N_{\la}$,
\begin{align*}
J_{\la}(u)=& \left(\frac{1}{p} -\frac{1}{r+1}\right)\|u\|^{p}- \la
\left(\frac{1}{q+1}-\frac{1}{r+1}\right)\int_{\Om}
h(x)|u|^{q+1} dx\\
\geq& c_1\|u\|_{X_0}^p - c_2 \|u\|_{X_0}^{q+1}.
\end{align*}
Hence $J_\la$ is bounded below and coercive on $\mc N_{\la}$.
\end{proof}
\section{Existence of solutions}
In this section, we show the existence of minimizers in $\mc
N_{\la}^{+}$ and $\mc N_{\la}^{-}$ for $\la\in(0,\la_0)$.

\begin{Lemma}\label{le31}
If $\la<\la_0$, then $J_{\la}$ achieve its minimum on $\mc
N_{\la}^{+}$.
\end{Lemma}

\begin{proof}
Since $J_\la$ is bounded below on $\mc N_{\la}$ and so on $\mc
N_{\la}^{+}$, there exists a minimizing sequence $\{u_k\}\subset
\mc N_{\la}^{+}$ such that
 \[\lim_{k\ra \infty}J_{\la}(u_k)= \inf_{u\in\mc N_{\la}^{+}}J_{\la}(u).\]
As $J_{\la}$ is coercive on $\mc N_{\la}$,  $\{u_k\}$ is a
bounded sequence in $X_0$. Therefore $u_k\rightharpoonup u_{\la}$
weakly in $X_0$ and $u_k \ra u_{\la}$ strongly in $L^{\al}(\mb R^n)$
for $\ds 1\leq \al<\frac{pn}{n-ps}$.

\noi If we choose $u\in X_0$ such that $\ds \int_{\Om} h(x)|u|^{q+1}
dx>0$, then there exist $t_1>0$ such that $t_1(u)u\in \mc
N_{\la}^{+}$ and $J_{\la}(t_{1}(u)u)<0$ and hence $\ds \inf_{u\in\mc
N_{\la}^{+}}J_{\la}(u)<0$. Now on $\mc N_{\la}$,
\[J_{\la}(u_k)= \left(\frac{1}{p}-\frac{1}{r+1}\right)\|u_k\|_{X_0}^p - \la\left(\frac{1}{q+1}-\frac{1}{r+1}\right)\int_{\Om} h(x)|u_k|^{q+1}dx\]
and so
\[\la\left(\frac{1}{q+1}-\frac{1}{r+1}\right)\int_{\Om} h(x)|u_k|^{q+1}dx =  \left(\frac{1}{p}-\frac{1}{r+1}\right)\|u_k\|^p- J_{\la}(u_k).\]
Letting $k \ra \infty$, we get $\int_{\Om} h(x)|u_\la|^{q+1} dx>0$.
Next we claim that $u_k \ra u_{\la}$. Suppose this is not true then
\[\int_{Q}|u_{\la}(x)-u_{\la}(y)|^p K(x-y) dxdy
<\liminf_{k\ra\infty} \int_{Q}|u_{k}(x)-u_{k}(y)|^p K(x-y) dxdy.\]
\noi Thus as
\[\phi_{u_k}^{\prime}(t) = t^{p-1}\|u_k\|_{X_0}^{p}- {\la
 t^{q}}\int_{\Om} h(x) |u_k|^{q+1} dx  - t^r\int_{\Om} b(x) |u_k|^{r+1}dx,\]
 and
\[ \phi_{u_{\la}}^{\prime}(t) = t^{p-1}\|u_{\la}\|_{X_0}^{p}- {\la t^{q}}\int_{\Om} h(x) |u_{\la}|^{q+1}dx - t^r\int_{\Om} b(x)
|u_{\la}|^{r+1}dx.\]
 It follows that $\phi_{u_k}^{\prime}(t_{\la}(u_\la))>0$ for sufficiently
large $k$. So, we must have $t_{\la}>1$ but $t_{\la}(u_{\la})
u_{\la}\in\mc N^{+}_{\la}$ and so
\[J_{\la}(t_{\la}(u_{\la}) u_{\la}) < J_{\la}(u_{\la})< \lim_{k\ra \infty} J_{\la}(u_k)= \inf_{u\in\mc N_{\la}^{+}}J_{\la}(u).\]
which is a contradiction. Hence we must have $u_k\ra u_{\la}$ in
$X_0$, and so $u_{\la}\in \mc N_{\la}^{+}$, since $\mc
N_{\la}^{0}=\emptyset$. Hence $u_{\la}$ is a minimizer for $J_{\la}$
on $\mc N_{\la}^{+}$.
 \QED
\end{proof}
\begin{Lemma}\label{le32}
If $\la<\la_0$, then $J_{\la}$ achieve its minimum on $\mc
N_{\la}^{-}$.
\end{Lemma}

\begin{proof}
Let $u\in \mc N_{\la}^{-}$ then from corollary \ref{co01}, we have
$J_{\la}(u)\geq \de_1$. So there exists a minimizing sequence
$\{u_k\}\subset \mc N_{\la}^{-}$ such that
 \[\lim_{k\ra \infty}J_{\la}(u_k)= \inf_{u\in\mc N_{\la}^{-}}J_{\la}(u)>0.\]
Since $J_{\la}(u_k)$ is coercive, $\{u_k\}$ is a bounded sequence
in $X_0$. Therefore $u_k\rightharpoonup u_{\la}$ weakly in $X_0$ and
$u_k \ra u_{\la}$ strongly in $L^{\al}$ for $1\leq
\al<\frac{np}{n-ps}$.
\[J_{\la}(u_k)= \left(\frac{1}{p}-\frac{1}{q+1}\right)\|u_k\|_{X_0}^p + \left(\frac{1}{q+1}-\frac{1}{r+1}\right)\int_{\Om} b(x)|u_k|^{r+1}dx.\]
Since $\ds \lim_{k\ra \infty}J_{\la}(u_k)>0$ and

\[ \lim_{k\ra\infty}\int_{\Om} b(x)|u_k|^{r+1} dx = \int_{\Om} b(x)|u_{\la}|^{r+1}
dx,\]
 we must have $\int_{\Om} b(x)|u_{\la}|^{r+1} dx>0$. Hence $\phi_{u_{\la}}$ has a global maximum at some point $\tilde{t}$ so that
$\tilde{t}(u_{\la}) u_{\la}\in \mc N_{\la}^{-}$. On the other hand,
$u_k\in \mc N_{\la}^{-}$ implies that $1$ is a global maximum point
for $\phi_{u_k}\; i.e. \; \phi_{u_k}(t)\leq \phi_{u_k}(1)$ for every
$t>0$. Thus we have
\begin{align*}
J_{\la}&(\tilde{t}(u_{\la}) u_{\la})\\&= \frac{1}{p}
(\tilde{t}(u_{\la}))^p\|u_{\la}\|^{p}_{X_0}- \frac{\la
 (\tilde{t}(u_{\la}))^{q+1}}{q+1}\int_{\Om} h(x) |u_{\la}|^{q+1} dx  - \frac{(\tilde{t}(u_{\la}))^{r+1}}{r+1}\int_{\Om} b(x)
 |u_{\la}|^{r+1}dx,\\
&<\liminf_{k\ra \infty}\left(\frac{1}{p}
(\tilde{t}(u_{\la}))^{p}\|u_{k}\|_{X_0}^{p}- \frac{\la
(\tilde{t}(u_{\la}))^{q+1}}{q+1}\int_{\Om} h |u_k|^{q+1}dx   -
\frac{(\tilde{t}(u_{\la}))^{r+1}}{r+1}\int_{\Om} b
 |u_k|^{r+1}dx\right),\\
&\leq\lim_{k\ra\infty}J_{\la}(\tilde{t}(u_{\la}) u_k)\leq
\lim_{k\ra \infty} J_{\la}(u_k)= \inf_{u\in \mc N^{-}_{\la}}
J_{\la}(u),
\end{align*}
which is a contradiction. Hence $u_k \ra u_{\la}$ and moreover
$u_{\la}\in \mc N^{-}_{\la}$, since $\mc N^{0}_{\la}=\emptyset$.\QED
\end{proof}

\noi Next, we prove the existence of non-negative solutions, for
this we first define some notations.
\[F_{+}=\int_{0}^{t}f_{+}(x,s) ds,\]
where $$ f_{+}(x,t)= \left\{
\begin{array}{lr}
f(x,t)\quad\mbox{if}\quad t\geq0\\
 0\quad\quad\quad\mbox{if}\quad t<0\\
\end{array}
\right.
$$
Let $J_{\la}^{+}(u)=\|u\|^{p}_{X_0} - \int_{\Om}F_{+}(x,u) dx $.
Then the functional $J_{\la}^{+}(u)$ is well defined and it is
Fr$\acute{e}$chet differentiable in $u\in X_0$ and for any $v\in
X_0$
\begin{align}\label{s1}
\ld J_{\la}^{\prime +}(u),v\rd= \int_{Q}
|u(x)-u(y)|^{p-2}&(u(x)-u(y))(v(x)-v(y))K(x-y) dxdy\notag\\
&\quad \quad\quad\quad\quad-\int_{\Om} f_{+}(x,u) v dx.
\end{align}
If $f(x,t):= \la h(x) |t|^{q-1}t +b(x) |t|^{r-1}t$. Then
$J_{\la}^{+}(u)$ satisfies all the above Lemmas. So for $\la\in
(0,\la_0)$ there exists two non-trivial critical points $u_{\la}\in
\mc N_{\la}^{+}$ and $v_{\la}\in \mc N_{\la}^{-}$.

\noi Now we claim that both $u_{\la}$ and $v_\la$ are non-negative in $\mb R^{n}$.
Take $v=u^{-}$ in \eqref{s1}, then
\begin{align*}
0=& \ld J_{\la}^{\prime +}(u), u^{-}\rd\\
=&\int_{Q} |u(x)-u(y)|^{p-2}(u(x)-u(y))(u^{-}(x)-u^{-}(y))K(x-y) dxdy\\
=&\int_{Q} |u(x)-u(y)|^{p-2}((u^{-}(x)-u^{-}(y))^2 + 2 u^{-}(x)u^{+}(y))K(x-y) dxdy\\
\geq& \int_{Q} |u^{-}(x)-u^{-}(y)|^{p}K(x-y) dxdy\\
=&\|u^{-}\|^{p}_{X_{0}}
\end{align*}
Thus $\|u^{-}\|_{X_{0}}=0$ and hence $u= u^{+}$. So by taking
$u=u_{\la}$ and $u=v_{\la}$ respectively, we get the non-negative
solutions of \eqref{eq01}.

{\bf {Proof of Theorem \ref{t1}}} Lemmas \ref{le31}, \ref{le32},
\ref{le30} and above discussion complete the proof.


\end{document}